\documentclass[12pt]{amsart}
\usepackage[latin1]{inputenc}	

\usepackage[english]{babel}
\usepackage{amsmath,amsthm}
\usepackage{amssymb,latexsym}
\usepackage{graphicx}

\newcommand{\C}{\mathbb{C}}
\newcommand{\K}{\mathbb{K}}

\newcommand{\N}{\mathbb{N}}
\newcommand{\Q}{\mathbb{Q}}

\newcommand{\Z}{\mathbb{Z}}

\newtheorem{theorem}{Theorem}[section]
\newtheorem{lemma}{Lemma}[section]

\newtheorem{corollary}{Corollary}[section]

\theoremstyle{remark}
\newtheorem{rem}{Remark}[section]

\allowdisplaybreaks

\begin{document}
\title[Multiple Legendre polynomials]{Multiple Legendre polynomials\\  in Diophantine approximation}

\author{{\sc Raffaele Marcovecchio}}
\address{Dipartimento di Matematica\\
Universit\`a di Pisa \\
Largo B.~Pontecorvo, 5\\
56127 Pisa\\
Italy}
\email{marcovec@mail.dm.unipi.it}

\date{}
\begin{abstract}
We construct a class of multiple Legendre polynomials and prove that they 
satisfy an Ap\'ery-like recurrence. We give new upper bounds of the approximation 
measures of logarithms of rational numbers by algebraic numbers of bounded 
degree. We prove e.g. that the nonquadraticity exponent of $\log 2$ is bounded 
from above by $12.841618...$, thus improving upon a recent result of the author. 
Our construction also yields some other known results.
\end{abstract}



\maketitle
\section{Introduction}
Recently, in~\cite{MarcovecchioLog2}, the author introduced three sequences 
$\{P_t(w)\}$, $\{Q_t(w)\}$, $\{R_t(w)\}$,  ($t=0,1,2,\dots$) of polynomials 
depending on five positive integers $E_1,\dots,E_5$ satisfying suitable inequalities. 
The construction of $P_t(w)$, $Q_t(w)$ and $R_t(w)$ required quite a long computation, 
using an idea of Sorokin (see (1) and (21) in~\cite{SorokinLogs}), essentially consisting 
in applying differentiations and multiplications by powers of $w$, namely:
\begin{equation}			\label{sorold}
D_{E_5 t} \circ w^{E_4 t} \circ D_{E_3 t} \circ w^{E_2 t} \circ D_{E_1 t}, 
\end{equation}
to the three functions 
\[
\frac{1}{1-w}, \qquad \frac{\log(1/w)}{1-w}\quad 
											\text{ and }\quad \frac{\frac{1}{2}\log^2(1/w)}{1-w}.
\]
Throughout the present paper $D_m (f(u))=\frac{1}{m!}(\frac{{\rm d}}{{\rm d} u})^m f(u)$. 
We obtained simultaneous approximations $P_t(w)\log(1/w)-Q_t(w)$ 
and $P_t(w)\log^2(1/w)-2R_t(w)$ to $\log(1/w)$ and $\log^2(1/w)$ 
(see (46) and (47) in~\cite{MarcovecchioLog2}). For suitable choices of $E_1,\dots,E_5$ 
not all equal, by applying the Rhin-Viola method (see~\cite{RhinViolaDilogs} and the 
bibliography therein) we proved that the coefficients of the polynomial $P_t(w)$ have 
a big common divisor $\Delta_t$, and similarly for $Q_t(w)$ and $R_t(w)$ (see 
Proposition 3.1 of~\cite{MarcovecchioLog2}). By making $w=a/b\in\Q$ we obtained 
new irrationality and nonquadraticity measures for logarithms of some rational 
numbers. Taking $w=1/2$ we proved 
\begin{equation}		\label{OldLog2}
\mu(\log2)<3.5745539\dots \quad \text{ and }\quad \mu_2(\log2)<15.6514202\dots\,.
\end{equation}
As usual (see e.g.~\cite{AmorosoViola} and~\cite{MarcovecchioViola}) for any complex 
number $\xi$, for any number field $\K$ and any integer $d\geq1$ we denote by 
$\mu_{d,\K}(\xi)$ the supremum of the exponents $\mu>0$ such that 
\[
0<|\xi-\alpha|<H(\alpha)^{-\frac{\mu d}{[\Q(\alpha):\Q]}}
\]
has infinitely many solutions $\alpha$ of the degree at most $d$ over $\K$. 
We abbreviate $\mu_{1,\K}(\xi)=\mu_\K(\xi)$, $\mu_{d,\Q}(\xi)=\mu_d(\xi)$ 
and $\mu_{1,\Q}(\xi)=\mu(\xi)$. By $H(\alpha)$ we mean the na\"ive height of $\alpha$, 
i.e. the maximum of the absolute values of the coefficients of its minimal polynomial over $\Z$. 
An upper bound for $\mu(\xi)$ (resp. for $\mu_2(\xi)$) is said to be an irrationality 
measure (resp. a nonquadraticity measure) of the irrational (resp. nonquadratic) number $\xi$. 
An upper bound for $\mu_\K(\xi)$ (resp. for $\mu_{2,\K}$) is said to be a $\K$--irrationality 
(resp. $\K$--nonquadraticity) measure of the number $\xi\notin\K$ (resp. $[\K(\xi):\K]>2$).

The aim of the present paper is to extend the results of~\cite{MarcovecchioLog2}) 
to obtain a lower bound of $\mu_d(\log w)$ for classes of rational numbers $w$ 
depending on $d$, for any $d\geq1$. Our Theorem 7.1 represents an improvement 
of Theorem 3 of~\cite{SorokinLogs}, in that for any $d\geq1$ it applies to a wider 
class of rational numbers $w$, and gives better bounds. We achieve this by producing 
bounds from below of the values in $\log w$ of polynomials $P$ of degree at most $d$, 
in terms of the na\"ive height $H(P)$. As a byproduct, we obtain a new nonquadraticity 
measure of $\log 2$, and a new proof of all the ($\K$--)irrationality and 
($\K$--)nonquadraticity measures proved in~\cite{MarcovecchioLog2} 
and~\cite{MarcovecchioViola}, without using the saddle point method, 
not even in its simplest form.

As to the proof, we extend Rukhadze's method (see~\cite{Rukhadzelow}) relying on  
Legendre-type polynomials, which yields the inequality $\mu(\log 2)<3.891399\dots$ (see 
the references in~\cite{MarcovecchioLog2} for other proofs of her result, and the references 
in~\cite{NesterenkoLog2} for accounts of earlier results). Rukhadze introduced the integrals
\[
\int_0^1 D_{6t} \big( y^{7t} (1-y)^{7t}\big) \frac{dy}{1+y} 
		= \int_0^1 \frac{y^{7t} (1-y)^{7t}}{(1+y)^{6t+1}}\ dy = s_t \log 2 - r_t, \; 
		s_t\in\Z, r_t\in\Q, 
\]
and proved that the coefficients of $D_{6t} \big( y^{7t} (1-y)^{7t}\big)$ have a non-trivial 
common divisor $\Delta_t$ such that $\lim_{t\to\infty}\Delta_t^{1/t}$ such that can be 
determined via the Prime Number Theorem. In order to unify Rukhadze's and Sorokin's 
approaches, we take advantage of a fundamental remark of Nesterenko 
(see~\cite{NesterenkoLog2}, Lemma 1). Indeed, if we make the substitution $z=w/(w-1)$ 
in the Sorokin-type rational functions
\[
 \sum_{k\geq0} \binom{k+p_1}{p_1+q_1} \cdots \binom{k+p_n}{p_n+q_n} w^k
\]
we obtain polynomials in $z$ (see (\ref{legendreHadamard}) below), and this leads us to a 
new class of multiple Legendre polynomials. We also combine two formulas for the so-called 
hyperharmonic numbers (see (\ref{DerivativeBinomial}) below). This allows us to reduce 
all the necessary analytic estimates of the relevant polynomials to the estimate of simple 
real integrals, as in Rukhadze's paper.

With the choice $z=-1$ (and hence $w=1/2$) such polynomials yield the new bound
\begin{equation}		\label{NewNonQuadrLog2}
\mu_2(\log2)<12.841618\dots\ .
\end{equation}
We use the approximations
\[
\int_0^1 D_{5t} \Big( y^{6t} (1-y)^{6t} D_{7t} \big( y^{5t} (1-y)^{5t} \big)\Big) 
		\frac{dy}{1+y}= \tilde{s}_t \log2 -\tilde{r}_t
\] 
to obtain the same irrationality measure of $\log2$ as in (\ref{OldLog2}).
As to (\ref{NewNonQuadrLog2}), we employ the polynomials
\[
D_{6t} \bigg( z^{7t} (1-z)^{7t} D_{8t} \Big( z^{9t} (1-z)^{9t} 
		D_{10t} \big( z^{7t} (1-z)^{7t}  \big)  \Big) \bigg)
\]
to construct simultaneous approximations to $\log\frac{z}{z-1}$ and $\log^2\frac{z}{z-1}$. 
It is conceivable to obtain (\ref{NewNonQuadrLog2}) through the Meyer $G$-function method 
of Nesterenko.

This paper is organized as follows. In section 2 we introduce a class of multiple Legendre 
polynomials and functions, whose symmetry properties are the counterpart of the integral 
identities obtained in~\cite{MarcovecchioLog2} as an application of the Rhin-Viola 
method. In section 3 we obtain analytic and arithmetic properties of these polynomials 
and functions by means of a combinatorial identity involving hyperharmonic numbers. 
In section 4 we prove that under suitable hypotheses our polynomials and functions 
satisfy an Ap\'ery-like recurrence. In section 5 we find the roots of the characteristic 
equation associated with the recurrence found in section 4. As a result, in the applications 
given in sections 7 and 8 we do not need to compute the actual coefficients of the recurrence. 
In all our numerical examples in section 8 some of the moduli of the roots of the characteristic 
polynomial coincide. Hence we need to apply to our recurrence a suitable extension of 
Poincar\'e's theorem, due to Pituk~\cite{Pituk} (see however Remark 8.1). In section 6 we 
prove the crucial arithmetical properties of the coefficients of our polynomials. In section 7 
we obtain our main result. To this end we need to generalize a Lemma due to Hata~\cite{HataPi}. 
In section 8 we obtain (\ref{NewNonQuadrLog2}), and indicate how to recover the nonquadraticity 
measures of logarithms proved by Hata~\cite{HataC2Saddle}, the author~\cite{MarcovecchioLog2}, 
and Viola and the author~\cite{MarcovecchioViola}. We also indicate briefly how to recover results of 
Heimonen, Matala-Aho and V\"a\"an\"anen~\cite{HeiMataVGauss}, Viola~\cite{ViolaHyperGeo}, 
Amoroso and Viola~\cite{AmorosoViola} and of Viola and the author~\cite{MarcovecchioViola}. 
We also give three more examples of our result:
\[
\mu_3\big(\log (5/4)\big)<67.9403\dots, \quad \mu_3\big(\log(6/5)\big)<37.963\dots,
\quad \mu_4\big(\log(20/19)\big)<565.5663\dots
\]
(compare with $\mu_3\big(\log(6/5)\big)<155$ in~\cite{SorokinLogs}).
\section{Multiple Legendre polynomials and functions}
Throughout this paper, the bijection
\[
\C\setminus\{1\}\ni z\mapsto w=\frac{z}{z-1}\in\C\setminus\{1\}
\]
plays an important role. We remark that:
\begin{enumerate}
\item $(1-w)(1-z)=1$;
\item $z\mapsto w$ is an involution;
\item $\Re z<1/2$ if and only if $|w|<1$;
\item $0\leq z<1$ if and only if $w\leq0$.
\end{enumerate}
The idea of using this bijection comes from Lemma 1 of~\cite{NesterenkoLog2}, 

For any $n\geq1$, let $p_1,\dots, p_n;q_1,\dots,q_n\geq0$ be integers, and let 
\[
{\mathcal L}_n (p_1,q_1; \dots; p_n,q_n;z)\in\Z[z]
\]
be the polynomials recursively defined for by
\[
{\mathcal L}_n (p_1,q_1; \dots; p_n,q_n;z) 
 = z^{q_n} (1-z)^{p_n}  D_{p_n+q_n} \big(z^{p_n} (1-z)^{q_n} 
				 		{\mathcal L}_{n-1} (p_1,q_1; \dots; p_{n-1},q_{n-1};z) \big),
\]
where by agreement ${\mathcal L}_0 (z)\equiv1$. Here and in the sequel 
$D_m (f(u))=\frac{1}{m!}(\frac{{\rm d}}{{\rm d} u})^m f(u)$. For example
\[
{\mathcal L}_1 (p_1,q_1;z)=(-z)^{q_1}(1-z)^{p_1},
\]
\[
{\mathcal L}_2 (p_1,q_1;p_2,q_2;z)=(-1)^{q_1} z^{q_2} (1-z)^{p_2} D_{p_2+q_2} 
\big( z^{p_2+q_1}(1-z)^{p_1+q_2} \big),
\]
\begin{multline*}
{\mathcal L}_3 (p_1,q_1;p_2,q_2;p_3,q_3;z) \\ 
		=(-1)^{q_1} z^{q_3} (1-z)^{p_3} D_{p_3+q_3} 
				\Big( z^{p_3+q_2}(1-z)^{p_2+q_3}  D_{p_2+q_2} 
				\big( z^{p_2+q_1}(1-z)^{p_1+q_2} \big)    \Big),
\end{multline*}
and so on. Let $M$ be the degree of ${\mathcal L}_n (p_1,q_1;\dots; p_n;q_n;z)$:
\[
M=\sum_{l=1}^n (p_l+q_l).
\]

Besides the multiple Legendre polynomials ${\mathcal L}_n (p_1,q_1; \dots; p_n,q_n;z)$, 
we shall need the multiple Legendre functions ${\mathcal I}_n^m (p_1,q_1; \dots; p_n,q_n;z)$, 
with $m=1,\dots,n-1$. They are defined, for $z\not=0,1$, by 
\begin{multline}	\label{defInmpqz}
{\mathcal I}_n^m (p_1,q_1;\dots;p_n,q_n;z):\\
		={\mathcal L}_n (p_1,q_1;\dots;p_n,q_n;z) \log^m\frac{z}{z-1} 
		- {\bf T}^m \big({\mathcal L}_n (p_1,q_1;\dots;p_n,q_n;z)\big),
\end{multline}
where ${\bf T}$ is the linear mapping
\[
{\bf T}:\C[z]\mapsto \C[z], \quad {\bf T}(P)(z):=\int\limits_0^1 \frac{P(z)-P(y)}{z-y},
\]
${\bf T}^m$ is the $m$-fold compositum ${\bf T}\circ\cdots\circ{\bf T}$, 
and $\log w=\log\frac{z}{z-1}$ is the principal value of the logarithm: 
$\log w=\log |w|+i \arg w$, with $-\pi<\arg w\leq\pi$. Our definition 
(\ref{defInmpqz}) naturally extends the classical notion of Legendre 
functions of the second kind, while our definition of 
${\bf T} ({\mathcal L}_n (p_1,q_1; \dots; p_n,q_n;z))$ extends the well-known  
notion of Christoffel's polynomials (see~\cite{Erdelyi}, vol. I, p.153, formula (26)).

Any polynomial $P(z)\in\C[z]$ equals a rational function in the new 
variable $w=\frac{z}{z-1}$, and our multiple Legendre polynomials 
${\mathcal L}_n (p_1,q_1; \dots; p_n,q_n;z)$ are closely related 
to the Sorokin-type rational functions ${\mathcal S}_n (p_1,q_1; \dots; p_n,q_n;w)$ 
recursively defined by
\[
{\mathcal S}_n (p_1,q_1; \dots; p_n,q_n;w) :=w^{q_n} 
		D_{p_n+q_n} \big(w^{p_n }{\mathcal S}_{n-1} (p_1,q_1; \dots; p_{n-1},q_{n-1};w)\big),
\]
where ${\mathcal S}_0(w)=\frac{1}{1-w}$.
\begin{lemma}
Let $a,b,c,d\geq0$ be integers such that $b+c=a+d$. Then 
\begin{equation}		\label{legenratspecial}
D_a \big( z^b (1-z)^c \big) = (-1)^{a+b} (1-w)^{a+1} D_a \bigg( \frac{w^b}{(1-w)^{d+1}} \bigg).
\end{equation}
\end{lemma}
\begin{proof}
Writing $z^b= \big(1-(1-z)\big)^b$ and applying the binomial theorem,
\[
D_a \big( z^b (1-z)^c \big) = 
	\sum_{l=\max\{0,a-c\}}^b (-1)^{a+l} \binom{b}{l} \binom{c+l}{a} (1-z)^{c+l-a}.
\]
Similarly, from $w^b= \big(1-(1-w)\big)^b$ we have 
\[
D_a \Big( \frac{w^b}{(1-w)^{d+1}} \Big) = \sum_{l=\max\{0,b-d\}}^b (-1)^{b+l} 
							\binom{b}{l} \binom{c+l}{a} \frac{1}{(1-w)^{c+l+1}}. \qedhere
\]
Comparing the two summations we get (\ref{legenratspecial}).
\end{proof}
\begin{lemma}
Let $P(z)\in\C[z]$, and let $R(w)$ be the rational function satisfying
\[
(1-z)P(z)=R(w).
\] 
Let $p,q\geq0$ be integers. Then
\begin{equation}			\label{Dspqwz}
z^q (1-z)^{p+1}  D_{p+q} \big(z^p (1-z)^q P(z) \big) = w^q D_{p+q} \big(w^p R(w)\big). 
\end{equation}
\end{lemma} 
\begin{proof}
By linearity, we may suppose $P(z)=(1-z)^k$, so that $R(w)=\frac{1}{(1-w)^{k+1}}$. 
In this particular case (\ref{Dspqwz}) is simply (\ref{legenratspecial}) 
with $a=p+q$, $b=p$, $c=k+q$ and $d=k$.
\end{proof}

We have by (\ref{Dspqwz})
\begin{equation}		\label{legendreHadamard}
 (1-z) {\mathcal L}_n (p_1,q_1;\dots; p_n,q_n;z) \\ 
 	= {\mathcal S}_n (p_1,q_1;\dots; p_n,q_n;w) .
\end{equation}
Since for any integer $j\geq0$
\[
(1-z)(1-z)^j=\frac{1}{(1-w)^{j+1}}=\sum\limits_{k=0}^\infty \binom{k+j}{j} w^k,
\]
and $\binom{k+j}{j}$ is a polynomial in $k$ of degree $j$, then for any polynomial $P(z)\in\C[z]$ 
there exists a polynomial $Q(k)\in\C[k]$, with $\deg P = \deg Q$, such that 
\[
(1-z)P(z) = \sum\limits_{k=0}^\infty Q(k) w^k,
\]
and the mapping $P\mapsto Q$ is bijective.
\begin{lemma}
Let $P(z)\in\C[z]$ and $Q(k)\in\C[k]$ as above. Let $p,q\geq0$ be integers. Then 
\begin{equation}			\label{derivativeBinomial}
z^q (1-z)^{p+1}  D_{p+q} \big(z^p (1-z)^q P(z) \big) = 
			\sum\limits_{k=0}^\infty \binom{k+p}{p+q} Q(k) w^k.
\end{equation}
\end{lemma}
\begin{proof}
This is a simple application of (\ref{Dspqwz}).
\end{proof}
By induction on $n$ we easily get 
\begin{equation}			\label{sorokin}
{\mathcal S}_n(p_1,q_1; \dots; p_n,q_n;w) := \sum_{k\geq0} \binom{k+p_1}{p_1+q_1}
			\cdots \binom{k+p_n}{p_n+q_n} w^k.
\end{equation}
In particular, 
\[
(p_1+q_1)!\cdots (p_n+q_n)! {\mathcal S}_n(p_1,q_1; \dots; p_n,q_n;w) 
\]
is a bisymmetric function of $p_1,\dots,p_n$ and $q_1,\dots,q_n$. Therefore 
\[
(p_1+q_1)!\cdots (p_n+q_n)! {\mathcal L}_n(p_1,q_1; \dots; p_n,q_n;z) 
\]
is also a bisymmetric function of $p_1,\dots,p_n$ and $q_1,\dots,q_n$. As a consequence
\[
{\rm ord}_{z=0}  {\mathcal L}_n(p_1,q_1; \dots; p_n,q_n;z) = \max_{l=1,\dots,n} q_l, 
\qquad 
{\rm ord}_{z=1}  {\mathcal L}_n(p_1,q_1; \dots; p_n,q_n;z) = \max_{l=1,\dots,n} p_l.
\]		
Moreover the polynomial 
\[
{\mathcal L}_n(p_1,q_1; \dots; p_n,q_n;z) 
\]
is a symmetric function of $(p_1,q_1),\dots,$ $(p_n,q_n)$, and we have
\[
{\mathcal L}_n(p_1,q_1; \dots; p_n,q_n;z) 
 = z^{q_1} (1-z)^{p_1}  D_{p_1+q_1} \big(z^{p_1} (1-z)^{q_1} 
				 		{\mathcal L}_{n-1}(p_2,q_2; \dots; p_n,q_n;z) \big),
\]
which from now on we define to be the standard recursive formula. From 
the definition (\ref{defInmpqz}) it also follows that the normalized multiple 
Legendre functions
\[
(p_1+q_1)!\cdots (p_n+q_n)!\, {\mathcal I}_n^m (p_1,q_1; \dots; p_n,q_n;z) 
\]
are bisymmetric functions of $p_1,\dots,p_n$ and $q_1,\dots,q_n$ 
as well. Let us notice that the equality
\[
(p_1+q_1)!(p_2+q_2)! {\mathcal L}_2(p_1,q_1; p_2,q_2;z)= 
(p_1+q_2)!(p_2+q_1)! {\mathcal L}_2(p_1,q_2; p_2,q_1;z)
\] 
is the content of (1.3) of~\cite{HataLegendre}, and of Lemma 3.1 of~\cite{RhinViolaDilogs}.

The polynomials ${\mathcal L}_n(p_1,q_1;\dots; p_n;q_n;z)$ also satisfy
\begin{equation}				\label{LnMirror}
{\mathcal L}_n(p_1,q_1; \dots; p_n,q_n;z)=(-1)^M {\mathcal L}_n(q_1,p_1; \dots; q_n,p_n;1-z),
\end{equation}
which plainly follows by repeated applications of
\begin{equation}			\label{derivzetaunomenzeta}		
\frac{{\rm d}}{{\rm d} z} f(z)= -\frac{{\rm d}}{{\rm d} y}\big(f(1-y)\big)_{y=1-z}.
\end{equation}
The lemma below provides us with an upper bound of the polynomials 
${\mathcal L}_n(p_1,q_1; \dots; p_n,q_n;z)$ in the unit interval. 
\begin{lemma}
Let $p_1,q_1;\dots,p_n,q_n\geq0$ be integers. Then
\begin{equation}		\label{boundLnzunitinter}
\max_{z\in[0,1]} z^{-q_1} (1-z)^{-p_1}|{\mathcal L}_n (p_1,q_1;\dots; p_n,q_n;z)| 
\leq \frac{M!}{(p_1+q_1)!\cdots(p_n+q_n)!}.
\end{equation}
\end{lemma}
\begin{proof}
For any $p,q\geq0$ and any $z\in[0,1]$ we have $|{\mathcal L}_1 (p,q;z)|=z^q (1-z)^p$, 
whence (\ref{boundLnzunitinter}) holds for $n=1$. Arguing by induction, let $n>1$ and 
suppose that (\ref{boundLnzunitinter}) is valid for $n-1$. Let $p=p_{n-1}$, $q=q_{n-1}$, 
$p^\prime=p_n$ and $q^\prime=q_n$. By Leibnitz's formula 
\[
D_{p+q} \big( z^{p+q^\prime} (1-z)^{p^\prime+q} \big) 
	= \sum_{k=0}^{p+q} D_{p+q-k} (z^{p+q^\prime}) D_k \big( (1-z)^{p^\prime+q} \big),
\]
whence
\begin{multline*}
{\mathcal L}_n (p_1,q_1;\dots; p_n,q_n;z)		\\
	= \sum_{k=0}^{p+q} \binom{p+q^\prime}{p+q-k} \binom{p^\prime+q}{k} 
	{\mathcal L}_{n-1} (p_1,q_1;\dots; p_{n-2},q_{n-2};p+q+p^\prime-k,q^\prime+k;z).
\end{multline*}
By the triangle inequality and the inductive hypothesis 
\[
z^{-q_1} (1-z)^{-p_1} |{\mathcal L}_n (p_1,q_1;\dots; p_n,q_n;z)|\leq \frac{M!}{(p_1+q_1)!\cdots(p_{n-2}+q_{n-2})!(p+q+p^\prime+q^\prime)!}\ L, 
\]
where by the Chu-Vandermonde formula 
\[
L= \sum_{k=0}^{p+q} \binom{p+q^\prime}{p+q-k} \binom{p^\prime+q}{k} 
= \binom{p+p^\prime+q+q^\prime}{p+q}. \qedhere
\]
\end{proof}
\section{Summation identities}
The following lemma will be useful.
\begin{lemma}
For any nonnegative integers $j,k$ 
\begin{equation}		\label{DerivativeBinomial}
\binom{k+j}{j}\sum\limits_{i=1}^j \frac{1}{k+i} = 
										\sum\limits_{i=1}^j \binom{k+j-i}{j-i} \frac{1}{i}.
\end{equation}
\end{lemma}
\begin{proof}
This follows from two formulas for the hyperharmonic numbers
\[
H_j^{(k+1)}=\binom{k+j}{j} \big( H_{k+j}^{(1)}-H_k^{(1)} \big)
		=\sum\limits_{i=1}^j \binom{k+j-i}{j-i} H_i^{(0)}, 
\]
where $H_j^{(0)}=\frac{1}{j}$, and recursively 
$H_j^{(k+1)}=H_1^{(k)}+\cdots+H_j^{(k)}$ 
(see~\cite{ConwayGuy}, p.258, and formulas (3) and (7) in~\cite{BenjaminGaeblers}).
\end{proof}

As in the previous section, let $P(z)\in\C[z]$ and $Q(k)\in\C[k]$ be such that
\[
(1-z)P(z) = \sum\limits_{k=0}^\infty Q(k) w^k
\]
for any $w,z\in\C$ satisfying $(1-z)(1-w)=1$. We claim that
\begin{equation}		\label{DerivativeNesterenko}
(1-z) {\bf T} (P)(z)= \sum\limits_{k=0}^\infty Q'(k) w^k, 
\end{equation}
where $Q'(k)=\frac{d}{dk} Q(k)$. Indeed, by linearity we may suppose that 
$P(z)=P_j(z)=(1-z)^j$ for some integer $j\geq0$, whence $Q(k)=Q_j(k)=\binom{k+j}{j}$ and
\[
(1-z){\bf T} (P_j)(z)=(1-z)\sum\limits_{i=1}^j \frac{(1-z)^{j-i}}{i}=
		\sum\limits_{i=1}^j \frac{1}{i}\sum_{k=0}^\infty \binom{k+j-i}{j-i} w^k.
\]
By (\ref{DerivativeBinomial}), the last summation equals
\[
\sum_{k=0}^\infty \binom{k+j}{j}\sum\limits_{i=1}^j \frac{1}{k+i} w^k =
		\sum\limits_{k=0}^\infty Q_j'(k) w^k,
\]
as we claimed. By induction on $m$ we also have by (\ref{DerivativeNesterenko})
\[
(1-z) {\bf T}^m (P)(z)= \sum\limits_{k=0}^\infty Q^{(m)}(k) w^k,
\]
where $Q^{(m)}(k)=\frac{d^m}{dk^m} Q(k)$.

Suppose now that $Q(-1)=0$. Then
\[
P(z)=(1-w)\sum\limits_{k=0}^\infty Q(k) w^k 
			= \sum\limits_{k=0}^\infty \big( Q(k)- Q(k-1) \big) w^k = (1-z) \tilde{P}(z)
\]
for some $\tilde{P}(z)\in\C[z]$. If in addition $Q^{(m)}(-1)=0$ then
\[
{\bf T}^m (P) (z)=(1-w)\sum\limits_{k=0}^\infty Q^{(m)}(k) w^k 
			= \sum\limits_{k=0}^\infty \big(Q^{(m)}(k) - Q^{(m)}(k-1) \big) w^k 
			= (1-z) {\bf T}^m (\tilde{P}) (z),
\]
and by linearity ${\bf T}^m \big((az+b)\tilde{P}(z)\big)=(az+b){\bf T}^m (\tilde{P})(z)$ 
for any $a,b\in\C$.

We may put $\tilde{Q}(k)=Q(k)- Q(k-1)$ and repeat the previous argument. Suppose 
that $Q(-1)=\cdots=Q(-l)=0$. Then $P(z)=(1-z)^l \hat{P}(z)$ for some $\hat{P}(z)\in\C[z]$. 
If also $Q^{(m)}(-1)=\cdots=Q^{(m)}(-l)=0$ 
then ${\bf T}^m (P) (z) = (1-z)^{l-j} {\bf T}^m \big((1-z)^j\hat{P}(z)\big)$ 
for $j=0,\dots,l$. Equivalently,
\begin{equation}			\label{orthogonality}
{\bf T}^m \big((a_l z^l+\cdots+a_0)\tilde{P}(z)\big) 
= (a_l z^l+\cdots+a_0){\bf T}^m (\tilde{P})(z)
\end{equation}
for any $a_0,\dots,a_l\in\C$.

We now apply (\ref{orthogonality}) to our multiple Legendre functions 
${\mathcal I}_n^m (p_1,q_1;\dots;p_n,q_n;z)$. Let
\[
{\mathcal L}_n^* (p_1,q_1; \dots; p_n,q_n;z) := (-1)^{q_1} z^{-q_1} (1-z)^{-p_1}
{\mathcal L}_n (p_1,q_1; \dots; p_n,q_n;z)\in\Z[z].
\]
We now set
\[
P(z)=(1-z)^{p_1+q_1} {\mathcal L}_n^* (p_1,q_1; \dots; p_n,q_n;z).
\] 
Taking into account that $\binom{k+p_1}{p_1+q_1}=0$ for $k=0,\dots,q_1-1$, 
by combining (\ref{legendreHadamard}) and (\ref{sorokin}) we get
\[
Q(k)=\prod_{j=1}^n \binom{k+p_j+q_1}{p_j+q_j} = \prod_{j=1}^n B_j(k),
\]
say. Moreover $Q(-1)=\cdots=Q(-p_1-q_1)=0$, because $B_1(k)$ vanishes at 
$-1,\dots,-p_1-q_1$. 

Let $n\geq2$. If
\[
p_1\leq \max\{p_2,\dots,p_n\} \text{ and } q_1\leq \max\{q_2,\dots,q_n\}
\]
then $Q'(-1)=\cdots=Q'(-p_1-q_1)=0$, because at least one among $B_2,\dots,B_n$
vanishes. Hence
\[
{\bf T} \big(W(z) {\mathcal L}_n^* (p_1,q_1; \dots; p_n,q_n;z) \big)  
	= W(z) {\bf T} \big( {\mathcal L}_n^* (p_1,q_1; \dots; p_n,q_n;z) \big) 
\]
for any $W(z)\in\C$ with $\deg W\leq p_1+q_1$. More generally, if
\[
p_1\leq p_2\leq \cdots\leq p_{m+1} \text{ and } q_1\leq q_2\leq \cdots\leq q_{m+1},
\]
for some $m\in\{1,\dots,n-1\}$, then $Q^{(m)}(-1)=\cdots=Q^{(m)}(-p_1-q_1)=0$, whence
\begin{equation}			\label{orthogonalityLnstar}
{\bf T}^m \big(W(z) {\mathcal L}_n^* (p_1,q_1; \dots; p_n,q_n;z) \big)  
	= W(z) {\bf T}^m \big( {\mathcal L}_n^* (p_1,q_1; \dots; p_n,q_n;z) \big) 
\end{equation}
for any $W(z)\in\C$ with $\deg W\leq p_1+q_1$. In particular
\[
{\bf T}^m \big({\mathcal L}_n (p_1,q_1; \dots; p_n,q_n;z) \big)  
																=  (-1)^{q_1} z^{-q_1} (1-z)^{-p_1} 
{\bf T}^m \big({\mathcal L}_n^* (p_1,q_1; \dots; p_n,q_n;z) \big).
\]

Since 
\[
	{\mathcal L}_1 (p,q;z)-{\mathcal L}_1 (p,q;y)=
		 (-z)^q (1-z)^p - (-z)^q (1-y)^p + (-z)^q (1-y)^p - (-y)^q (1-y)^p
\]
then
\[
{\bf T} \big({\mathcal L}_1 (p,q;z) \big) = \sum\limits_{i=1}^p \frac{1}{i}{\mathcal L}_1 (p-i,q;z)
+\sum\limits_{i=1}^q {\mathcal L}_1 (0,q-i;z) \int\limits_0^1 (-y)^{i-1} (1-y)^p\ dy,
\]
where
\[
d_{p+i} \int\limits_0^1 (-y)^{i-1} (1-y)^p\ dy \in\Z \quad \text{ and } \quad
			\int\limits_0^1 y^{i-1} (1-y)^p\ dy<1.
\]
Here and in the sequel, for any positive integer $l$ we denote by $d_l$ 
the least common multiple of $1,\dots,l$, and we put $d_0=1$. 

On the other hand
\[
(1-z){\bf T} \big({\mathcal L}_1 (p,q;z) \big) = 
	\sum\limits_{k=0}^\infty  \frac{{\rm d}}{{\rm d} k}  \Bigg( \binom{k+p}{p+q}\Bigg) w^k.
\]
By (\ref{sorokin}), (\ref{DerivativeNesterenko}) and Lebnitz's formula, 
\begin{multline*}
(1-z){\bf T} \big({\mathcal L}_2 (p_1,q_1;p_2,q_2;z) \big) \\
= \sum_{k=0}^\infty \frac{{\rm d}}{{\rm d} k}
			\Bigg( \binom{k+p_1}{p_1+q_1}\Bigg) \binom{k+p_2}{p_2+q_2} w^k 
+ \sum_{k=0}^\infty \binom{k+p_1}{p_1+q_1} 
			\frac{{\rm d}}{{\rm d} k} \Bigg( \binom{k+p_2}{p_2+q_2}\Bigg) w^k.
\end{multline*}
Thus, by (\ref{derivativeBinomial})
\begin{multline*}
{\bf T} \big({\mathcal L}_2 (p_1,q_1;p_2,q_2;z) \big) = \\
	  \sum\limits_{i=1}^{p_1} {\mathcal L}_2(p_1-i,q_1;p_2,q_2;z) \frac{1}{i} 
	+\sum\limits_{i=1}^{q_1} {\mathcal L}_2(0,q_1-i;p_2,q_2;z) 
			\int\limits_0^1 (-y)^{i-1} (1-y)^{p_1}\ dy\\
	+\sum\limits_{i=1}^{p_2} {\mathcal L}_2(p_1,q_1;p_2-i,q_2;z) \frac{1}{i} 
	+\sum\limits_{i=1}^{q_2} {\mathcal L}_2(p_1,q_1;0,q_2-i;z) 
			\int\limits_0^1 (-y)^{i-1} (1-y)^{p_2}\ dy.
\end{multline*}
Similarly, for any $n\geq1$ the Legendre function
${\bf T} \big({\mathcal L}_n (p_1,q_1;\dots; p_n,q_n;z)\big)$ is a finite sum of terms of the form 
$\frac{1}{r}{\mathcal L}_n (p_1^\prime,q_1^\prime;\dots; p_n^\prime,q_n^\prime;z)$, 
with $p_j^\prime\leq p_j$ and $q_j^\prime\leq q_j$ ($j=1,\dots,n$), and 
$r+p_s^\prime+q_s^\prime\leq p_s+q_s$ for some $s=1,\dots,n$. Hence 
$d_{\max\{p_1+q_1,\dots,p_n+q_n\}}{\bf T} \big({\mathcal L}_n (p_1,q_1;\dots; p_n,q_n;z)$
has integer coefficient. 

Clearly, for any $m=1,\dots,n-1$ we may iterate the above decomposition $m$ times.  
Hence ${\bf T}^m \big({\mathcal L}_n (p_1,q_1;\dots; p_n,q_n;z)\big)$ is a finite sum 
of terms of the type $\frac{1}{r_1\cdots r_m}
{\mathcal L}_n (p_1^\prime,q_1^\prime;\dots; p_n^\prime,q_n^\prime;z)$, and the set 
$\{r_1,\dots,r_m\}$ is the disjoint union of subsets $\{r_1^\prime,\dots,r_j^\prime\}$, each 
of which satisfies $r_1^\prime+\cdots+r_j^\prime\leq p_s+q_s$, for distinct $s=1,\dots,n$. 
Let $H_1,\dots,H_n$ be a reordering of 
\[
p_1+q_1,\dots,p_n+q_n
\] 
such that 
\[
H_1\geq\cdots\geq H_n.
\]
We have 
\begin{equation}			\label{arithmTmnztrivial}
d_{H_1} d_{\max\{H_2,[\frac{H_1}{2}]\}}\cdots d_{\max\{H_m,[\frac{H_1}{m}]\}}
{\bf T}^m \big({\mathcal L}_n (p_1,q_1;\dots; p_n,q_n;z)\big) \in\Z[z]
\end{equation}
because the product $r_1^\prime\cdots r_j^\prime$ 
divides the product $d_{p+q}\cdots d_{[\frac{p+q}{j}]}$ 
whenever $r_1^\prime+\cdots+r_j^\prime\leq p+q$. 

In addition, if $p_1\leq p_2\leq \cdots\leq p_m$ and $q_1\leq q_2\leq \cdots\leq q_m$, 
by (\ref{boundLnzunitinter}) and the triangle inequality
\[
\max_{z\in[0,1]} z^{-q_1} (1-z)^{-p_1}
		|{\bf T}^{i-1} \big({\mathcal L}_n (p_1,q_1;\dots; p_n,q_n;z) \big)| 
				\leq M^{i-1}\, \frac{M!}{(p_1+q_1)!\cdots(p_n+q_n)!}
\]
for any $i=1,\dots,m$. If $z\in\C\setminus[0,1]$ then
\begin{multline*}
{\mathcal I}_n^m (p_1,q_1;\dots; p_n,q_n;z)\\
= \sum_{i=1}^m \log^{m-i} \frac{z}{z-1} \Big( 
		{\bf T}^{i-1} \big( {\mathcal L}_n (p_1,q_1;\dots; p_n,q_n;z)\big) \log \frac{z}{z-1} - 
		{\bf T}^i \big( {\mathcal L}_n (p_1,q_1;\dots; p_n,q_n;z)\big)		\Big)		\\
= \sum_{i=1}^m \log^{m-i} \frac{z}{z-1} 
	\int\limits_0^1 \frac{{\bf T}^{i-1} \big( {\mathcal L}_n (p_1,q_1;\dots; p_n,q_n;y)\big)}{z-y}.
\end{multline*}
It follows that
\[
|{\mathcal I}_n^m (p_1,q_1;\dots; p_n,q_n;z)|\leq 
		C(m,z) M^{m-1}\, \frac{M!}{(p_1+q_1)!\cdots(p_n+q_n)!} 
				\max\limits_{y\in[0,1]} y^{q_1}(1-y)^{p_1},
\]
where $C(m,z)$ is a constant depending only on $m$ and $z$. Since
\[
\max\limits_{y\in[0,1]} y^{q_1}(1-y)^{p_1} = 
			\frac{p_1^{p_1}q_1^{q_1}}{(p_1+q_1)^{p_1+q_1}},
\]
we obtain
\begin{equation}			\label{boundImnz}
|{\mathcal I}_n^m (p_1,q_1;\dots; p_n,q_n;z)| 
			\leq C(m,z) M^{m-1}\, \frac{M!}{(p_1+q_1)!\cdots(p_n+q_n)!} 
				\frac{p_1^{p_1}q_1^{q_1}}{(p_1+q_1)^{p_1+q_1}},
\end{equation}
The above upper bound is far from being sharp, nevertheless it is useful in the applications of 
Pituk's theorem in section 5.
\section{Linear recurrence relations}
For any $p,q\geq0$ integers, let ${\mathcal D}_{p,q}$ be the differential operator defined by
\[
{\mathcal D}_{p,q} \big(P(z)\big) =  z^q (1-z)^p D_{p+q} \big(z^p (1-z)^q P(z) \big)
\]
for any polynomial $P(z)\in\C[z]$. For $p_1,p_2, q_1,q_2\geq0$ integers, 
${\mathcal D}_{p_1,q_1}$ and  ${\mathcal D}_{p_2,q_2}$ commute, 
as is clear for $P(z)=z^{q_3}$ from the symmetry properties 
of ${\mathcal L}_3(p_1,q_1; p_2,q_2; 0,q_3;z)$, and hence for any $P(z)$ by linearity.
\begin{lemma}
Let $k,p,q\geq0$ be integers. Then for any $P(z)\in\C[z]$
\begin{equation}		\label{abeliankernel}
z^k (1-z)^{-k} {\mathcal D}_{p+k,q} \big(P(z) \big) = 
		{\mathcal D}_{p,q+k} \big( z^k (1-z)^{-k} P(z) \big) .
\end{equation}
\end{lemma}
\begin{proof}
This is just a formal application of the definition of ${\mathcal D}_{p,q}$.
\end{proof}
\begin{lemma}
Let $p,p^\prime,q,q^\prime$ be integers such that $p\geq p^\prime\geq0$, $q\geq q^\prime\geq0$. 
Then for any polynomial $P(z)\in\Q[z]$ there exists a polynomial $R(z)\in\Q[z]$ such that
\begin{equation}		\label{kernel}
{\mathcal D}_{p,q} \big( P(z) \big) = {\mathcal D}_{p^\prime,q^\prime} \big( R(z) \big).
\end{equation}
\end{lemma}
\begin{proof}
This plainly follows from (\ref{derivativeBinomial}).
\end{proof}
Now we are ready to prove the following
\begin{theorem}
Let $p_1,\dots,p_n;q_1,\dots,q_n$ be integers such that 
$p_1,\dots,p_n>0$ and $q_1,\dots,q_n\geq0$. Then there exist 
$n+1$ polynomials $G_0(t,z),\dots,G_n(t,z)\in\Z[t,z]$ such that
\
\begin{equation}			\label{nontriviality}
\big(G_0(t,\zeta),\dots,G_n(t,\zeta)\big)\in\big(\C[t]\big)^{n+1}\setminus\{(0,\dots,0)\}
\quad \text{ for any } \zeta\in\C
\end{equation}
and 
\begin{equation}		\label{recurrenceLn}
\sum_{l=0}^n  G_l(t,z) {\mathcal L}_n \big(p_1(t+l),q_1(t+l);\dots; p_n(t+l),q_n(t+l);z \big) =0.
\end{equation}
Furthermore, if $n\geq2$ and if
\begin{equation}			\label{monotonicityppqqs}
p_1\leq p_2\leq \cdots\leq p_{m+1} \text{ and } q_1\leq q_2\leq \cdots\leq q_{m+1},
\end{equation}
where $1\leq m<n$, then
\begin{equation}		\label{recurrenceTmLn}
\sum_{l=0}^n G_l(t,z)  
		{\bf T}^m \Big({\mathcal L}_n \big(p_1(t+l),q_1(t+l);\dots; p_n(t+l),q_n(t+l);z \big) \Big)=0.
\end{equation}
\end{theorem}
\begin{proof}
Let
\[
{\mathcal H}_t(z) := {\mathcal L}_n \big(p_1 t,q_1 t;\dots; p_n t,q_n t;z \big) =
{\mathcal D}_{p_1 t, q_1 t} \circ\cdots\circ {\mathcal D}_{p_n t, q_n t} (1).
\] 
We claim that for any $m$ independent of $t$ the coefficient of $z^{Mt-m}$ in 
${\mathcal H}_t(z)$ is
\[
\frac{(M t)!}{\big((p_1+q_1) t\big)!\cdots\big((p_n+q_n) t\big)!}\, R(t),
\]
where $R(t)$ is a rational function in $\Q(t)$. Indeed, by (\ref{legendreHadamard}), 
(\ref{sorokin}) and (\ref{derivzetaunomenzeta}), and using the bijection $P\mapsto Q$ 
described in section 3, this follows from the fact that for any $m^\prime=0,\dots,m$ 
the coefficient of $k^{Mt-m^\prime}$ in 
\[
\binom{k+p_1 t}{(p_1+q_1)t} \cdots \binom{k+p_n t}{(p_n+q_n)t} 
\big((p_1+q_1) t\big)!\cdots\big((p_n+q_n) t\big)!
\]
is a polynomial in $t$, and the coefficient of $k^{Mt-m^\prime}$ in
\[
\binom{k+M t-m^\prime}{M t-m^\prime} (M t)!
\]
is a rational function in $t$.

Let $L$ and $t_0$ be integers satisfying 
\begin{equation}		\label{largeL}
 L>\frac{Mn(n-1)}{2}-n
\end{equation}
and 
\begin{equation}		\label{larget0}
 t_0>\frac{L+Mn}{\min\{p_1,\dots,p_n\}}.
\end{equation}
Let $A_0(z),\dots, A_n(z)\in\Z[z]$ satisfy $\deg A_l (z) \leq L+M(n-l) \quad (l=0,\dots,n)$. 
Then for any integer $t\geq 0$ the degree of the polynomial $A_l(z) {\mathcal H}_{t+l}(z)$ 
does not exceed $L+M(t+n)$ $(l=0,\dots,n)$. By solving in $\Q(t)$ a system of 
homogeneous linear equations we can find $A_0(z)=G_0(t;z),\dots,A_n(z)=G_n(t;z)\in\Z[t,z]$ 
not all zero such that for any integer $t\geq0$ the polynomial
\begin{equation}		\label{null}
  G_0(t;z) {\mathcal H}_t (z) + G_1(t;z) {\mathcal H}_{t+1}(z)+\cdots
+G_n(t;z) {\mathcal H}_{t+n}(z)
\end{equation}
has degree not exceeding
\begin{equation}		\label{upper}
L+M(t+n)-\sum_{l=0}^n (L+Ml+1) +1
\end{equation}
(by agreement the degree of $0$ is $-\infty$). 
Since $\max\limits_{l=0,\dots,n} \deg_z F_l(t,z)\leq L+Mn$, and using the binomial 
theorem in the form
\begin{equation}				\label{binomialtheoremAB}
z^A = \sum_{j=A}^B \binom{B-A}{j-A} z^j (1-z)^{B-j}, \quad 0\leq A\leq B,
\end{equation}
there exist $\alpha_{j,l}(t)\in\Z[t]$ ($l=0,\dots,n; j=0,\dots,L+Mn)$ such that
\[
G_l(t,z)=\sum_{j=0}^{L+Mn} \ \alpha_{j,l} (t) z^j (1-z)^{L+Mn-j} \quad (l=0,\dots,n).
\]
For any $t\geq t_0$, with $t_0$ as in (\ref{larget0}), by $n$ applications of (\ref{abeliankernel})
\[
z^j (1-z)^{-j} {\mathcal H}_{t+l}(z)  = {\mathcal D}_{p_1(t+l)-j,q_1(t+l)+j} \circ\cdots\circ  
			{\mathcal D}_{p_n(t+l)-j,q_n(t+l)+j} \big( z^j (1-z)^{-j} \big).
\]
We remark that
\[
{\mathcal D}_{p_n(t+l)-j,q_n(t+l)+j} \big( z^j (1-z)^{-j} \big) 
	= (-1)^j {\mathcal D}_{p_n(t+l)-j,q_n(t+l)+j} \big( 1 \big).
\]
By repeated applications of (\ref{kernel}), 
for any $t\geq t_0$ there exists $R_{j,l,t}(z)\in\Q[z]$ such that 
\[
{\mathcal D}_{p_1(t+l)-j, q_1(t+l)+j} \circ\cdots\circ  
			{\mathcal D}_{p_n (t+l)-j,q_n (t+l)+j} ( 1 )	
	= {\mathcal D}_{p_1 t-L-Mn,q_1 t} \circ\cdots\circ  
			{\mathcal D}_{p_n t-L-Mn, q_n t} \big( R_{j,l,t}(z) \big).
\]
It follows that the sum in (\ref{null}) equals
\begin{equation}		\label{nullRz}
(1-z)^{L+Mn}  {\mathcal D}_{p_1 t-L-Mn,q_1 t} \circ\cdots\circ  
		{\mathcal D}_{p_n t-L-Mn, q_n t} \big( N_t(z) \big),
\end{equation}
with
\[
N_t(z)=\sum_{l=0}^n \sum_{j=0}^{L+Mn}\ (-1)^j \alpha_{j,l} (t) R_{j,l,t}(z).
\]
If (\ref{null}) (i.e. (\ref{nullRz})) was non-zero, then $N_t(z)$ would be also 
non-zero, whence the degree of the polynomial (\ref{nullRz}) would be not 
exceeding (\ref{upper}) but at least
\begin{equation}		\label{lower}
L+Mn+\sum_{l=1}^n (p_l t+q_l t -L-Mn).
\end{equation}
This is impossible, because by (\ref{largeL}) the integer (\ref{upper}) is strictly lower 
than (\ref{lower}) for any $t$. Hence (\ref{null}) vanishes at least for any $t\geq t_0$. 
For $m=0,\dots, m_0=L+M(t_0+n)$ the coefficient of $z^{L+M(t+n)-m}$ 
in (\ref{null}) is a rational function in $\Q(t)$, vanishing at infinitely many $t$'s, 
therefore it vanishes identically. It follows that (\ref{null}) vanishes for any $t\geq0$. 
 
If 
$G_0(t,\zeta),\dots,G_n(t,\zeta)$ are all zero for some $\zeta\in\C$, putting 
\[
\hat{G}_0(t,z)=\frac{G_0(t,z)}{z-\zeta}, \dots, \hat{G}_n(t,z)=\frac{G_n(t,z)}{z-\zeta},
\]
we have that $\hat{G}_0(t,z),\dots,\hat{G}_n(t,z)$ are not all zero, and
\[
\sum_{l=0}^n  \hat{G}_l(t,z) 
{\mathcal L}_n \big(p_1(t+l),q_1(t+l);\dots; p_n(t+l),q_n(t+l);z \big) = 0.
\]
By repeating this argument finitely many times we eventually find $n+1$
polynomials $G_0(t,z)$, $\dots,G_n(t,z)$ satisfying (\ref{nontriviality})  
and (\ref{recurrenceLn}).

The equality (\ref{recurrenceLn}) is plainly equivalent to
\[
\sum_{l=0}^n G_l(t,z)  z^{q_1 l} (1-z)^{p_1 l} 
			{\mathcal L}_n^* \big(p_1(t+l),q_1(t+l);\dots; p_n(t+l),q_n(t+l);z \big) = 0,
\]
which implies 
\[
\sum_{l=0}^n {\bf T}^m \Big( G_l(t,z)  z^{q_1 l} (1-z)^{p_1 l} 
			{\mathcal L}_n^* \big(p_1(t+l),q_1(t+l);\dots; p_n(t+l),q_n(t+l);z \big) \Big) = 0,
\]
by linearity. If (\ref{monotonicityppqqs}) holds, then by (\ref{orthogonalityLnstar})
\[
\sum_{l=0}^n G_l(t,z)  z^{q_1 l} (1-z)^{p_1 l} {\bf T}^m \Big(  
			{\mathcal L}_n^* \big(p_1(t+l),q_1(t+l);\dots; p_n(t+l),q_n(t+l);z \big) \Big) = 0
\]
for any (sufficiently large) $t$. Again by (\ref{orthogonalityLnstar}) we get (\ref{recurrenceTmLn}).
\end{proof}
\begin{rem}
There is no essential restriction in supposing that $p_1,\dots,p_n$ 
are strictly  positive. Indeed, since ${\mathcal L}_1 (0,0;z)=1$ and
\[
{\mathcal L}_2 (0,q;p,0;z)=\binom{p+q}{p} {\mathcal L}_1 (p,q;z),
\]
by Theorem 2.1 one may suppose $p_j+q_l>0$ for $j,l=1,\dots,n$, 
i.e. either $p_1,\dots,p_n>0$ or $q_1,\dots,q_n>0$. By (\ref{LnMirror}),
up to changing $z$ into $1-z$ one may suppose $p_1,\dots,p_n>0$.

Moreover, if (\ref{monotonicityppqqs}) holds, then by (\ref{defInmpqz}), (\ref{recurrenceLn}) 
and (\ref{recurrenceTmLn}) we also have
\begin{equation}			\label{recurrenceImnz}
\sum_{l=0}^n  G_l(t,z) {\mathcal I}_n^m \big(p_1(t+l),q_1(t+l);\dots; p_n(t+l),q_n(t+l);z \big) =0.
\end{equation}
\end{rem}
\section{Roots of the characteristic polynomial}
We now prove that (\ref{recurrenceLn}) is a Poincar\'e recurrence, and find its characteristic 
equation. Let $G_0(t,z),\dots,G_n(t,z)$ be polynomials in $\Z[t,z]$ satisfying (\ref{nontriviality})
and (\ref{recurrenceLn}). For $\zeta\in\C$ the degree of $G_i(t,\zeta)$ with respect to $t$ may 
depend on $\zeta$. Let
\[
\delta_\zeta=\max\{\deg_t G_0(t,\zeta),\dots,\deg_t G_n(t,\zeta)\}.
\]
By (\ref{nontriviality}) $\delta_\zeta\geq0$ (let us recall that the degree of $0$ is $-\infty$). 
Let 
\begin{equation}			\label{glzeta}
g_l(\zeta)=\lim_{t\to\infty} \frac{G_l(t,\zeta)}{t^{\delta_\zeta}} \quad (l=0,\dots,n). 
\end{equation}
Then $\delta_\zeta=\deg_t G_l(t,\zeta)$ and $g_l(\zeta)\not=0$ for some $l=0,\dots,n$.
\begin{theorem}
Let $p_1,\dots,p_n>0$ and $q_1,\dots,q_n\geq0$ be integers. Let $v,y\in\C$ 
and $\zeta\in\C\setminus\{0,1\}$ such that
\[
\zeta\prod_{j=1}^n (y+p_j) = (\zeta-1) \prod_{j=1}^n (y-q_j)\not=0 \quad \text{ and }\quad
v=\prod_{j=1}^n \frac{(y-q_j)^{q_j} (y+p_j)^{p_j}}{(p_j+q_j)^{p_j+q_j}},
\]
and let $g_0(\zeta),g_1(\zeta),\dots,g_n(\zeta)$ be as in (\ref{glzeta}).

Then
\begin{equation}		\label{roots}
g_n(\zeta) v^n+g_{n-1}(\zeta) v^{n-1}+\cdots+g_0(\zeta)=0.
\end{equation}
\end{theorem}
\begin{proof}
By (\ref{binomialtheoremAB}) we have
\[
G_l(t,z)=\sum_{r=0}^s \beta_{l,r} (t) z^r (1-z)^{s-r}, \quad l=0,\dots,n,
\]
for suitable $\beta_{l,r}(t)\in\Z[t]$. We put
\[
G_l^*(t,w):=\sum_{r=0}^s (-1)^{s-r} \beta_{l,r}(t) w^r, \quad l=0,\dots,n.
\]
For any $z,w\in\C$ with $(1-z)(1-w)=1$ we have $G_l(t,z)=(1-z)^s G_l^*(t,w)$. 
Also, if $\Re z<\frac{1}{2}$ (or equivalently $|w|<1$),
by (\ref{sorokin}), (\ref{legendreHadamard}) and (\ref{recurrenceLn})
\begin{equation}		\label{recurrenceseries}
\sum_{l=0}^n G_l^*(t,w)  \sum_{k\geq \overline{q}(t+l)} w^k 
									\prod_{j=1}^n \binom{k+p_j (t+l)}{(p_j+q_j)(t+l)} = 0.
\end{equation}
Thus the coefficients of all powers of $w$ in (\ref{recurrenceseries}) vanish, i.e.
\[
\sum_{l=0}^n \sum_{r=0}^s \beta_{l,r}^* (t) 
	\prod_{j=1}^n \binom{k-r+p_j (t+l)}{(p_j+q_j)(t+l)} = 0
\]
for all integers $k\geq0$ and $t\geq0$, 
with $\beta_{l,r}^*(t)=(-1)^{s-r} \beta_{l,r} (t)$. For convenience 
we change $k$ into $k+s$, and interchange $r$ and $s-r$, so that
\begin{equation}		\label{binomialidentity}
\sum_{l=0}^n \sum_{r=0}^s \beta_{l,s-r}^*(t) \prod_{j=1}^n 
		\binom{k+r+p_j (t+l)}{(p_j+q_j)(t+l)} = 0.
\end{equation}
For any fixed $t\geq0$, the sum in (\ref{binomialidentity}) is a polynomial in $k$ 
having infinitely many roots, and therefore it must be identically zero. It follows 
that (\ref{binomialidentity}) holds for any $k\in\C$ and any integer $t\geq0$. In addition, 
\[
\binom{k+r+p(t+l)}{(p+q)(t+l)} = 
		\binom{k+p t}{(p+q) t} \frac{(k-q(t+l)+1)_{q l} (k+p t+1)_{r+p l}}
		{(k-q (t+i)+1)_r ((p+q)t+1)_{(p+q)i}}.
\]
As usual, the Pochhammer symbol is defined by $(\xi)_0=1$ and 
$(\xi)_m=\xi(\xi+1)\cdots(\xi+m-1)$ for any integer $m>0$. From (\ref{binomialidentity}) 
we also have 
\[
\sum_{l=0}^n \sum_{r=0}^s \beta_{l,s-r}^* (t) 
\prod_{j=1}^n \frac{(k-q_j (t+l)+1)_{q_j l} (k+p_j t+1)_{r+p_j l}}
{(k-q_j (t+l)+1)_r ((p_j+q_j)t+1)_{(p_j+q_j)l}} =0
\]
for any $k\in\C$, provided that $(k-q_j t+1)_{s-q_j l}$ is non-zero for $j,l=0,\dots,n$. 
Taking $k=yt$, this implies
\[
\sum_{l=0}^n \sum_{r=0}^s \beta_{l,s-r}^* (t) 
\prod_{j=1}^n \frac{((y-q_j )t-q_j l+1)_{q_j l} ((y+p_j) t+1)_{r+p_j l}}
{((y-q_j)t-q_j l+1)_r ((p_j+q_j)t+1)_{(p_j+q_j)l}} =0
\]
for any $t\in\C$, provided that $((y-q_j)t+1)_{s-q_j l}$ is non-zero for $j,l=0,\dots,n$. 
On multiplying by $t^{-\delta_\zeta}$ and making $|t|\to\infty$ 
\[
\sum_{l=0}^n \sum_{r=0}^s \beta_{l,s-r}^{**} 
\prod_{j=1}^n \frac{(y-q_j)^{q_j l-r} (y+p_j)^{r+p_j l}}{(p_j+q_j)^{(p_j+q_j)l}} =0,
\]
where $\beta_{l,s-r}^{**} = [t^{\delta_\zeta}] \beta_{l,s-r}^* (t) $ is the coefficient 
of $t^{\delta_\zeta}$ in $\beta_{l,s-r}^* (t)$. This is just the same as 
\[
\sum_{l=0}^n \sum_{r=0}^s \beta_{l,r}^{**} \xi^r v^l =0,
\]
where $\xi=\zeta/(\zeta-1)$. We conclude the proof by noticing that
\[
g_l(\zeta)=\sum_{r=0}^s \beta_{l,r}^{**} \zeta^r (\zeta-1)^{s-r}
\]
because $\beta_{l,r}^{**}= (-1)^{s-r} \lim_{t\to\infty} t^{-\delta_\zeta} \beta_{l,r}(t)$.
\end{proof}

Before discussing the applications of the above theorems, 
let us state some open questions. The generating function
\[
\Phi_1(p,q;v,w)=\sum_{t=0}^\infty  {\mathcal S}_1 (p t,q t;w) v^t
			= \frac{(1-w)^{p+q-1}}{(1-w)^{p+q}-vw^q}
\]
satisfies the differential equations
\[
\frac{\frac{\partial}{\partial v} \Phi_1(p,q;v,w)}{\Phi_1(p,q;v,w)}
		=\frac{\partial}{\partial v}\log \Phi_1(p,q;v,w) = \frac{w^q}{(1-w)^{p+q}-vw^q}
\]
and 
\[
\frac{\frac{\partial}{\partial w} \Phi_1(p,q;v,w)}{\Phi_1(p,q;v,w)}
		=-\frac{p+q-1}{1-w} + \frac{(p+q)(1-w)^{p+q-1}-qvw^{q-1}}{(1-w)^{p+q}-vw^q}.
\]
Hence by a theorem of Lipshitz (see Remark (2) on page 377 in~\cite{LipshitzDfin}) 
for any positive integer $n$ the series
\[
\Phi_n(p_1,q_1;\dots;p_n,q_n;v,w):= \sum_{t=0}^\infty \sum_{k=0}^\infty 
				\binom{k+p_1t}{(p_1+q_1)t}\cdots\binom{k+p_nt}{(p_n+q_n)t} v^t\, w^k \\
\]
satisfies a linear differential equation, with polynomial coefficients, containing only 
powers of $\frac{\partial}{\partial v}$ (and another one containing only powers of 
$\frac{\partial}{\partial w}$, but we do not need the last one). Equivalently, the 
series ${\mathcal S}_n (p_1t,q_1t;\dots; p_nt,q_nt;w)$ satisfies a linear recurrence 
relation with coefficients in $\Z[t,w]$. However, it is unclear to us whether the order 
of the recurrence obtained by this method would be $n$ or larger (see Corollary 4.4 
below for the minimality of the order of the recurrence (\ref{recurrenceLn}). 
Moreover, by a theorem of Sharif and Woodcock~\cite{WoodShaPro} the series 
$\Phi_2(p_1,q_1;p_2,q_2;v,w)$ is algebraic, and we know its explicit formula in 
the classical case $p_1=p_2=1$ and $q_1=q_2=0$:
\[
\Phi_2(1,0;1,0;v,w)= \sum_{t=0}^\infty 
						\sum_{k=0}^\infty \binom{k+t}{t}^2 v^t\, w^k 
						= \frac{1}{\sqrt{(1-v-w)^2-4vw}} \\
\]
(see e.g. Remark 5.2 on page 527 in~\cite{WoodShaPro}). It would be interesting to 
obtain a general expression for $\Phi_2(p_1,q_1;p_2,q_2;v,w)$ and to establish whether 
$\Phi_n(p_1,q_1;\dots;p_n,q_n;v,w)$ is algebraic or not, for $n\geq3$ (here we assume 
that the ground field has characteristic zero). Finally, it would be interesting to find an 
explicit version of (\ref{binomialidentity}), beyond the special case $n=2$, $p_1=p_2=1$ 
and $q_1=q_2=0$ considered above, and here it is perhaps useful to remark that
\[
\big(\Phi_2(1,0;1,0;v^2,w^2)\big)^2 
	= \Phi_1(1,0;v,w)\,  \Phi_1(1,0;v,-w)\, \Phi_1(1,0;-v,w)\, \Phi_1(1,0;-v,-w).
\]

We shall apply a theorem of Pituk (see (2.1) and (3.32) in~\cite{Pituk}):
\begin{theorem}[Pituk]
Let $n\geq1$ be a integers, and let $f,A_1,\dots,A_n:\N\to\C$ be 
sequences of complex numbers such that for any sufficiently large $t$
\[
f(t+n) + A_1(t) f(t+n-1) +\cdots+ A_n(t) f(t) =0,
\]
and 
\[
\lim_{t\to\infty} A_l (t) = a_l\in\C \quad (l=1,\dots,n).
\]
Then either $f(t)=0$ for any sufficiently large $t$, or
\begin{equation}		\label{theoPituk}
\lim_{t\to\infty} \sqrt[t]{\max\{|f(t)|,\dots,|f(t+n-1)|\}} = |\lambda|,
\end{equation}
where $\lambda\in\C$ satisfies 
\[
\lambda^n+a_1 \lambda^{n-1}+\cdots+a_n=0.
\]
\end{theorem}
We get the following
\begin{corollary}
Let $p_1,\dots,p_n>0$ and $q_1,\dots,q_n\geq0$ be integers, and 
let $\zeta\in\C\setminus[0,1]$. Let $y_1,\dots,y_n$ be the roots of 
\begin{equation}		\label{zrootsy}
\zeta \prod_{j=1}^n (y+p_j)= (\zeta-1)\prod_{j=1}^n (y-q_j) .
\end{equation}
If
\begin{equation}		\label{virootsy}
v_h=\prod_{j=1}^n \frac{(y_h-q_j)^{q_j}(y_h+p_j)^{p_j}}{(p_j+q_j)^{p_j+q_j}}
\end{equation}
are all distinct for $h=1,\dots,n$, then the recurrence (\ref{recurrenceLn}) has minimal order, and
\begin{equation}		\label{asymptoticLn}
\lim_{t\to\infty} \max_{l=0,\dots,n-1} 
		|{\mathcal L}_n (p_1(t+l),q_1(t+l);\dots; p_n(t+l),q_n(t+l);\zeta)|^{1/t} = |v_h|
\end{equation}
for some $h=1,\dots,n$.

Furthermore, if $\zeta$ is algebraic, $n\geq2$, $p_1\leq p_2\leq \cdots\leq p_{m+1}$ 
and $q_1\leq q_2\leq \cdots\leq q_{m+1}$, where $1\leq m<n$, then
\begin{equation}		\label{asymptoticlImn}
\lim_{t\to\infty} \max_{l=0,\dots,n-1} 
		|{\mathcal I}_n^m (p_1(t+l),q_1(t+l);\dots; p_n(t+l),q_n(t+l);\zeta)|^{1/t} = |v_k|
\end{equation}
for some $v_k$ satisfying
\[
|v_k|\leq 
\frac{p_1^{p_1} q_1^{q_1} M^M}{(p_1+q_1)^{2(p_1+q_1)}(p_2+q_2)^{p_2+q_2}
					\cdots(p_n+q_n)^{p_n+q_n}}.
\]
\end{corollary}
\begin{proof}
Let $g_0(\zeta),\dots,g_n(\zeta)$ be as in (\ref{glzeta}). By (\ref{roots}) the non-zero polynomial
\[
g_n(\zeta) v^n+g_{n-1} v^{n-1}+\cdots+g_0(\zeta)\in\C[v]
\]
has the $n$ distinct roots $v_1,\dots,v_n$ defined in (\ref{virootsy}), therefore 
it has degree $n$. This proves that (\ref{recurrenceLn}) has minimal order.

Furthermore $g_n(\zeta)\not=0$, whence $G_n(t,\zeta)\not=0$ for any sufficiently 
large $t\in\N$. Let us divide (\ref{recurrenceLn}) by $G_n(t,\zeta)$. By repeated 
applications of the mean value theorem, all the roots of 
${\mathcal L}_n (p_1t,q_1t;\dots; p_nt,q_nt;z)$ 
belong to $[0,1]$. Since $\zeta\notin[0,1]$ we may apply (\ref{theoPituk}) with 
\[
f(t)={\mathcal L}_n (p_1t,q_1t;\dots; p_nt,q_nt;\zeta)\not=0,
\]
and obtain (\ref{asymptoticLn}), 

As to (\ref{asymptoticlImn}), by Stirling's formula
\[
\lim_{t\to\infty} \left(\frac{(Mt)!}{((p_1+q_1)t)!\cdots((p_n+q_n)t)!}\right)^{1/t}=
		\frac{M^M}{(p_1+q_1)^{p_1+q_1}\cdots(p_n+q_n)^{p_n+q_n}}.
\]
Moreover ${\mathcal I}_n^m (p_1(t+l),q_1(t+l);\dots; p_n(t+l),q_n(t+l);\zeta)\not=0$ 
for all $t$, because $\log\frac{\zeta}{\zeta-1}$ is trascendental and 
${\mathcal L}_n (p_1(t+l),q_1(t+l);\dots; p_n(t+l),q_n(t+l);\zeta)$ is non-zero. As above 
we may divide (\ref{recurrenceImnz}) by $G_n(t,\zeta)$ and apply (\ref{theoPituk}) with 
\[
f(t)={\mathcal I}_n^m (p_1t,q_1t;\dots; p_nt,q_nt;\zeta)\not=0.
\]
By (\ref{boundImnz}) we get (\ref{asymptoticlImn}).
\end{proof}
\section{Arithmetical properties}
Let $K_1,\dots,K_{n^2}$ be a reordering of 
\[
p_1+q_1,\dots,p_1+q_n,\dots,p_n+q_1,\dots,p_n+q_n
\] 
such that 
\[
K_1\geq\cdots\geq K_{n^2}.
\]
Let $N_1,\dots,N_n$ be defined by
\[
N_1=K_1,\quad, N_2=\max\bigg\{K_2,\frac{K_1}{2}\bigg\}, \dots, 
	N_n=\max\bigg\{K_n,\frac{K_1}{n}\bigg\}.
\]
For any prime number $s\geq2$ and any positive integer $m$ we shall 
denote by $\nu_s(m)$ the largest power $a$ such that $s^a$ divides $m$.

For any real number $x$ we shall denote by $[x]$ the largest 
integer $m$ satisfying $m\leq x$. Let also $\{x\}=x-[x]$.

For any $\sigma\in\mathfrak{S}_n:=\{\tau:\{1,\dots,n\}\to\{1,\dots,n\} \text{ bijective} \}$ 
we shall write for brevity
\[
q_j^\sigma:=q_{\sigma(j)}\quad (j=1,\dots,n).
\]
For any $\omega\in [0,1)$ let
\[
\mu(\omega):=\max_{\sigma\in\mathfrak{S}_n} \sum_{j=1}^n 
		\big( [(p_j+q_j^\sigma)\omega] - [(p_j+q_j)\omega] \big)\geq0.
\]
For any integer $t\geq0$ and any prime number $s>\sqrt{N_1 t}$, let
\[
\Delta_t:=\prod_{\substack{s>\sqrt{N_1 t}\\ s \text{ prime}}} s^{\mu(\omega)}\in\Z, \quad
\text{ where } \omega=\{t/s\}:=t/s-[t/s].
\]
\begin{theorem}
Let $n\geq2$ be an integer, and let $m\in\{1,\dots,n-1\}$. With the above notation we have 
\begin{equation}		\label{arithmTmnzstrong}
\Delta_t^{-1} d_{N_1 t} \cdots d_{N_m t} 
	{\bf T}^m \big({\mathcal L}_n (p_1 t,q_1 t;\dots; p_n t,q_n t;z)\big) \in\Z[z] .
\end{equation}
\end{theorem} 
\begin{proof}
For any $\sigma\in\mathfrak{S}_n$
\[
{\bf T}^m \big({\mathcal L}_n (p_1 t,q_1 t;\dots; p_n t,q_n t;z)\big) =
{\bf T}^m \big({\mathcal L}_n (p_1 t,q_1^\sigma t;\dots; p_n t,q_n^\sigma t;z)\big)
		\prod_{j=1}^n \frac{\big( (p_j+q_j)t\big)!}{\big( (p_j+q_j^\sigma)t\big)!}.
\]
By (\ref{arithmTmnztrivial}), for any prime number $s>\sqrt{N_1 t}$ such that  
\[
\mu^\sigma (\omega) = \sum_{j=1}^n \big( [(p_j+q_j^\sigma)\omega] - [(p_j+q_j)\omega] \big)>0,
\]
where $\omega=\{t/s\}$, we have
\[
s^{-\mu^\sigma (\omega)} 
d_{N_1 t} \cdots d_{N_m t} 
{\bf T}^m \big({\mathcal L}_n (p_1 t,q_1 t;\dots; p_n t,q_n t;z)\big) \in\Z[z].
\]
The Lemma follows.
\end{proof}
We notice that by the Prime Number Theorem
\begin{equation}			\label{PNTheorem}
\lim_{t\to\infty} \frac{1}{t} \log d_{N_j t} = N_j\quad (j=1,\dots,n).
\end{equation}
Also, $[0,1)=\cup_{i=1}^m [u_i,u_{i+1})$ for some rational numbers 
$0=u_1<u_2<\cdots<u_{m+1}=1$ such that $\mu(\omega)$ is constant on each interval 
$[u_i,u_{i+1})$, and by a standard lemma (see e.g.~\cite[Lemma 6]{NesterenkoLog2})
\begin{equation}		\label{Delta}
\delta:=\lim_{t\to\infty} \frac{1}{t} \log \Delta_t = 
		\sum_{h=1}^m \mu(u_h)\big(\psi(u_{h+1})-\psi(u_h)\big),
\end{equation}
where $\psi(x)=\Gamma^\prime(x)/\Gamma(x)$ is the logarithmic derivative of Euler's 
gamma function. 
\section{The main result}
\begin{lemma}
Let $h<0$ be an integer. Let $\gamma_1,\dots,\gamma_l$ be real numbers 
such that $1,\gamma_1,\dots,\gamma_l$ are linearly independent over $\Q$. 
Let $\{r_{1,t}\}_{t\in\N},\dots,\{r_{l,t}\}_{t\in\N},\{s_t\}_{t\in\N}$ be 
sequences of elements in $\Z+\sqrt{h}\,\Z$, and 
let $\sigma,\tau>0$ such that
\[
\lim\limits_{t\to\infty} \frac{1}{t} \log |s_t|=\sigma, \qquad 
\limsup\limits_{t\to\infty} \frac{1}{t} \log |s_t\gamma_j-r_{j,t}|	\leq 	-\tau\ (j=1,\dots,l).
\]
Then for any $a_1,\dots,a_l,b\in\Z$ and any $\varepsilon>0$
\[
|a_1 \gamma_1+\cdots+a_l \gamma_l -b|\gg_\varepsilon 
				(\max\{|a_1|,\dots,|a_l|\})^{-\frac{\sigma}{\tau}-\varepsilon}.
\]
\end{lemma}
\begin{proof}
We shall use the following notation:
\[
\Lambda=a_1 \gamma_1+\cdots+a_l \gamma_l - b,
\]
\[
\varepsilon_{j,t}=s_t\gamma_j-p_{j,t} \quad (j=1,\dots,l),
\]
\[
A_t=a_1 r_{1,t}+\cdots+a_l r_{l,t} + b s_t,
\]
\[
{\bf a}=(a_1,\dots,a_l), \quad \|{\bf a}\|=\max\{|a_1|,\dots,|a_l|\}.
\]
Throughout this proof $b$ is the integer closest to $a_1 \gamma_1+\cdots+a_l \gamma_l$. 
Arguing by contradiction, suppose that there exists $\eta>\sigma/\tau$ such that 
$|\Lambda|\leq \|{\bf a}\|^{-\eta}$ for infinitely many vectors ${\bf a}\in\Z^l$. 

For any $\varrho>0$ there exist two constants $c_1(\varrho), c_2(\varrho)>0$ 
such that for any $t\geq0$
\[
c_1(\varrho) e^{(\sigma-\varrho)t}\leq |s_t|\leq c_2(\varrho) e^{(\sigma+\varrho)t}, 
\qquad |\varepsilon_{j,t}|\leq c_2(\varrho) e^{-(\tau-\varrho)t}\ (j=1,\dots,l).
\]
For any ${\bf a}\in\Z^l$ satisfying $|\Lambda|\leq \|{\bf a}\|^{-\eta}$, and for any 
$\varrho\in(0,\tau)$, let $t({\bf a})$ be the smallest integer $t$ such that 
$c_2(\varrho) e^{-(\tau-\varrho)t} \|{\bf a}\| l<1/2$, i.e. 
\[
2lc_2(\varrho)<\frac{e^{(\tau-\varrho)t({\bf a})}}{\|{\bf a}\|}\leq 2lc_2(\varrho) e^{\tau-\varrho}.
\]
In particular $t({\bf a})\to\infty$ for $\|{\bf a}\|\to\infty$. For any $t\geq t({\bf a})$ 
\[
|a_1 \varepsilon_{1,t}+\cdots+a_l \varepsilon_{l,t}|		\leq 
		l \|{\bf a}\| c_2(\varrho) e^{-(\tau-\varrho)t} <\frac{1}{2}.
\]
On the other hand, since $\sigma-\eta\tau<0$ then $(\sigma+\varrho)-\eta(\tau-\varrho)<0$ 
for any sufficiently small $\varrho>0$, hence
\[
|s_{t({\bf a})} \Lambda|\leq |s_{t({\bf a})}| \|{\bf a}\|^{-\eta} 		\leq
		c_2(\varrho) e^{(\sigma+\varrho)t({\bf a})} 
		\Bigg(\frac{e^{(\tau-\varrho)(t({\bf a})-1)}}{2lc_2(\varrho)}\Bigg)^{-\eta}<\frac{1}{2}
\]
for any sufficiently large $\|{\bf a}\|$. Since
\[
s_t \Lambda = A_t + a_1 \varepsilon_{1,t}+\cdots+a_l \varepsilon_{l,t},
\]
we have $|A_{t(\bf a)}|<1$, hence $A_{t(\bf a)}=0$, for infinitely many 
$\bf a$'s. For any fixed $\bf a$ we have $A_t\not=0$ for infinitely many $t$, 
because $\varepsilon_{j,t}/s_t\to0$ for $t\to\infty$ ($j=1,\dots,l$) and $\Lambda\not=0$. 
Hence for any $\bf a$ there exists $\hat{t}({\bf a})>t({\bf a})$ such that 
\[
A_{t(\bf a)}=\cdots=A_{\hat{t}(\bf a)-1}=0\not=A_{\hat{t}(\bf a)}.
\]
We may suppose
\[
\frac{e^{(\tau-\varrho) \hat{t} ({\bf a})}}{\|{\bf a}\|} \to\infty \quad (\|{\bf a}\|\to\infty),
\]
otherwise $A_{\hat{t}(\bf a)}$ would vanish for any 
sufficiently large $\|{\bf a}\|$, just as above. Then
\[
s_{\hat{t}({\bf a})-1} \Lambda 
	= a_1 \varepsilon_{1,\hat{t}({\bf a})-1}+\cdots+a_l \varepsilon_{l,\hat{t}({\bf a})-1} \to0 
																							\quad (\|{\bf a}\|\to\infty).
\]
Then also $s_{\hat{t}({\bf a})} \Lambda \to0$ for $\|{\bf a}\|\to\infty$, whence 
$A_{\hat{t}(\bf a)}=0$ for any sufficiently large $\|{\bf a}\|$, which gives a contradiction.
\end{proof}
\begin{rem}
For $l=1,2$ the above Lemma is due to Hata (see Lemma 2.1 and Remark 2.1 
in~\cite{HataPi}, and Lemma 2.3 and Remark 1 in~\cite{HataC2Saddle}). 
Also, in the Lemma above and in the Corollary below one may replace 
$\max\{|a_1|,\dots,|a_l|\}$ by $\max\{|a_1|,\dots,|a_l|,|b|\}$ (see 
e.g. Remark 2.5 in~\cite{MarcovecchioViola}). 
\end{rem}
\begin{corollary}
Let $h<0$ and $n\geq1$ be integers. Let $\gamma_1,\dots,\gamma_l$ be real 
numbers such that $1,\gamma_1,\dots,\gamma_l$ are linearly independent over $\Q$. 
Let $\{r_{1,t}\}_{t\in\N},\dots,\{r_{l,t}\}_{t\in\N},\{s_t\}_{t\in\N}$ be sequences of 
elements in $\Z+\sqrt{h}\,\Z$, and let $\sigma,\tau>0$ such that
\[
\lim\limits_{t\to\infty} \frac{1}{t} \log \max\limits_{i=0,\dots,n-1}|s_{t+i}|=\sigma, \qquad 
\lim\limits_{t\to\infty} \frac{1}{t} \log \max\limits_{i=0,\dots,n-1}|s_t\gamma_j - r_{j,t+i}| 
																										=-\tau\ (j=1,\dots,l).
\]
Then for any $a_1,\dots,a_l,b\in\Z$ and any $\varepsilon>0$
\[
|a_1 \gamma_1+\cdots+a_l \gamma_l - b|\gg_\varepsilon 
											(\max\{|a_1|,\dots,|a_l|\})^{-\frac{\sigma}{\tau}-\varepsilon}.
\]
\end{corollary}
\begin{proof}
There exists a sequence $\{m_t\}_{t\in\N}$ of integers such that $t\leq m_t< t+n$ and
\[
|s_{m_t}|=\max\limits_{i=0,\dots,n-1}|s_{t+i}|.
\]
Let $s_t^\prime=s_{m_t}$ and $r_{j,t}^\prime=r_{j,m_t}$ ($j=1,\dots,l$). We have 
\[
\lim\limits_{t\to\infty} \frac{1}{t} \log |s_t^\prime|=\sigma, \qquad 
\limsup\limits_{t\to\infty} \frac{1}{t} \log |s_t^\prime\gamma_j - r_{j,t}^\prime|	
																									\leq -\tau\ (j=1,\dots,l).
\]
We now are in the same situation as in the above Lemma. 
\end{proof}
\begin{rem}
If $\gamma_i=\gamma^i$ ($i=1,\dots,l$), where $\gamma$ is a transcendental number, then 
we may write the conclusion of the above corollary in the following way: for every polynomial 
$P$ with integer coefficients and degree not exceeding $l$, and any $\varepsilon>0$
\begin{equation}			\label{lowboundPgammaheight}
|P(\gamma)|\gg_\varepsilon H(P)^{-\frac{\sigma}{\tau}-\varepsilon},
\end{equation}
where $H(P)$ is the na\"ive height of $P$. By standard arguments this implies 
\begin{equation}			\label{lowboundgammabetaheight}
|\gamma-\beta|\gg_\varepsilon H(\beta)^{-\frac{\sigma}{\tau}-1-\varepsilon}
\end{equation}
for any algebraic number $\beta$ of degree at most $l$. Here $H(\beta)$ 
is the na\"ive height of the minimal polynomial of $\beta$ over $\Z$.
\end{rem}
We now are ready to state our main result. Let $z=\zeta=a/b$ be a negative 
rational number, with $a<0<b$ and $|a|,b$ coprime. Let $n\geq2$, and let 
$p_1,\dots,p_n>0$ and $q_1,\dots,q_n\geq0$ be integers. Let $m$ be an integer 
such that $1\leq m<n$, and suppose that (\ref{monotonicityppqqs}) holds. 
Let $v_1,\dots,v_n$ be defined by (\ref{zrootsy}) and (\ref{virootsy}). Let 
\[
V=\max\{|v_1|,\dots,|v_n|\},
\]
and let 
\[
W=\max\Bigg\{|v_i|: i\in\{1,\dots,n\} \text{ and } |v_i|\leq 
	\frac{p_1^{p_1} q_1^{q_1} M^M}{(p_1+q_1)^{2(p_1+q_1)}(p_2+q_2)^{p_2+q_2}
																			\cdots(p_n+q_n)^{p_n+q_n}}\Bigg\}.
\]
Let $\delta$ be as in (\ref{Delta}). Let
\[
-\tau:=\log W -q_1\log|z|-p_1\log|1-z|+N_1+\cdots+N_m-\delta +(M-p_1-q_1)\log b
\]
and 
\[
\sigma:=\log V -q_1\log|z|-p_1\log|1-z|+N_1+\cdots+N_m-\delta +(M-p_1-q_1)\log b
\]
\begin{theorem}
With the above notation, if $\tau,\sigma>0$ then for any polynomial 
$P(x)\in\Z[x]$ of degree at most $m$ and any $\varepsilon>0$
\[
P\Big(\log\frac{z}{z-1}\Big)\gg_\varepsilon H(P)^{-\frac{\sigma}{\tau}-\varepsilon}.
\]
In particular, for any algebraic number $\beta$ of degree at most $m$
\[
\Big|\log\frac{z}{z-1}-\beta\Big|\gg_\varepsilon H(\beta)^{-\frac{\sigma}{\tau}-1-\varepsilon}.
\]
\end{theorem}
\begin{proof}
We apply our Corollary 7.1 to
\[
s_t=b^{(M-p_1-q_1)t} \Delta_t^{-1} d_{N_1 t}\cdots d_{N_m t}\  
		{\mathcal L}_n^*(p_1 t, q_1 t; \dots;p_n t, q_n t; z),
\]
\[
r_{j,t}=b^{(M-p_1-q_1)t} \Delta_t^{-1} d_{N_1 t}\cdots d_{N_m t} 
	z^{-q_1 t} (1-z)^{-p_1 t}\ {\mathcal I}_n^j (p_1 t, q_1 t; \dots;p_n t, q_n t; z) 
	\quad (j=1,\dots,m),
\]
and use (\ref{asymptoticLn}), (\ref{asymptoticlImn}), (\ref{arithmTmnzstrong}), 
(\ref{PNTheorem}) and (\ref{Delta}).
\end{proof}
We may also apply our Corollary 7.1 with $z^\prime=1-z=a^\prime/b$, with $a^\prime=b-a$, 
so that $0<a^\prime<b$, and by (\ref{LnMirror}) we obtain the same bound as in Theorem 7.1.

The above Theorem improves Theorem 3 of~\cite{SorokinLogs}.
\section{Numerical examples}
We recover the upper bound for $\mu(\log 2)$ we obtained in~\cite{MarcovecchioLog2}. 
Let $z=-1$, $n=3$, $p_1=4$, $p_2=5$, $p_3=3$, $q_1=1$, 
$q_2=2$, $q_3=0$. The roots of $-(y+3)(y+4)(y+5)=-2y(y-1)(y-2)$ 
are $20.267669670594\dots$ and $-1.133834835297\dots\pm i  1.294140012477\dots$. 
The values of $v_1,v_2,v_3$ in Corollary 4.3 are all distinct, 
and we have $\log|v_1|=22.149699678920\dots$, $\log|v_2|=\log|v_3|=-7.537440405644\dots$. 
We have $W=|v_2|$ because
\[
\log|v_2|< 4\log4+15\log15-3\log3-10\log5-7\log7=13.1543\dots<\log|v_1|. 
\]
The function $\mu(\omega)$ defined in section 6 takes the value $1$ if 
$\omega\in[1/6,3/7)\cup[1/2,5/7)\cup[3/4,6/7)$ and $\mu(\omega)=0$ 
otherwise, whence $\delta=4.995102335817...$. Thus by Theorem 7.1 
we have (see Theorem 1.1 in~\cite{MarcovecchioLog2})
\[
\mu(\log 2)\leq 3.574553902525...\ .
\]
\begin{rem}
Since $|v_1|>1>|v_2|=|v_3|$, we may also argue differently. By a theorem of Buslaev 
(see Theorem 1 in~\cite{Buslaev}), and by  (3) in~\cite{ZudilinDifference},
\[
\lim\limits_{t\to\infty}\frac {1}{t} \log|{\mathcal L}_3 (4t,t;5t,2t;3t,0;-1)|
	=\log|v_1|=22.149699...
\]
because ${\mathcal L}_3 (4t,t;5t,2t;3t,0;-1)\not=0$ is an integer, and
\[
\limsup\limits_{t\to\infty}\frac {1}{t} \log|{\mathcal I}_3^1 (4t,t;5t,2t;3t,0;-1)|
		=\log|v_2|-4\log 2=-10.310029...\ .
\]
Since $\log2$ is irrational, then the sequences 
\[
\{{\mathcal L}_3^* (4t,t;5t,2t;3t,0;-1)\}_{t\in\N} \quad \text{ and } \quad
\{z^{-t}(1-z)^{-4t} {\mathcal I}_3^1 (4t,t;5t,2t;3t,0;-1)\}_{t\in\N}
\]
are linearly independent. Hence may apply Theorem 1 of~\cite{ZudilinDifference}. 
\end{rem}

The special case $m=1$ and $n=2$ of our Theorem 7.1 is a reformulation of Theorem 2 
of~\cite{HeiMataVGauss} (here $q_1=0$, $q_2=p_2-p_1$, our $z$ is $s/r$ and $p_1/p_2$ 
is $\alpha$ in the notation of~\cite{HeiMataVGauss}), also proved, with a different method, 
in~\cite{ViolaHyperGeo} ($p_1=j-l$, $p_2=h$, $q_1=0$ and $q_2=l$ with the notation 
of~\cite{ViolaHyperGeo}). 

The results in the previous sections can be combined with Lemma 2.4 of~\cite{AmorosoViola} to 
obtain, again with $m=1$ and $n=2$, another proof of Theorem 2.3 of~\cite{AmorosoViola} 
(by choosing again $p_1=j-l$, $p_2=h$, $q_1=0$ $q_2=l$, and now $w=\alpha$, 
i.e. $z=\alpha/(\alpha-1)$ with the notation thereof), , without using the saddle point method. 
Also, by combining them with Lemma 2.2  of~\cite{MarcovecchioViola} one may obtain, 
with $m=1$, $n=3$ and $p_1=4$, $p_2=5$, $p_3=3$, $q_1=1$, $q_2=2$ and $q_3=0$, 
a new proof of all the new $\Q(\alpha)$--irrationality measures of $\log\alpha$ in (1.3) 
of~\cite{MarcovecchioViola}. We point out that Theorem 2.3 of~\cite{AmorosoViola} 
with $h=j=7$, $l=1$ gives
\[
\mu_{\Q\left(\sqrt{\frac{2}{3}}\right)} \left(\log\frac{2}{3}\right)<9.583516\dots\ . 
\] 

We now obtain our new nonquadraticity measure of $\log2$. We apply Theorem 7.1 
with $z=-1$, $m=2$, $n=4$, $p_1=5$, $p_2=6$, $p_3=7$, $p_4=4$, 
$q_1=1$, $q_2=2$, $q_3=3$, $q_4=0$. The roots 
of $-(y+4)(y+5)(y+6)(y+7)=-2y(y-1)(y-2)(y-3)$ 
are $-1.730890645949...$, $38.527585178736...$ 
and $-1.398347266393...\pm i 3.262020254989...\ $. The values of $v_1$, $v_2$, $v_3$ 
and $v_4$ in Corollary 4.3 are all distinct, and  $\log|v_1|=48.947490848559...$, 
$\log|v_2|=\log|v_3|=-9.276132199490\dots$, $\log|v_4|=-18.115257059384...$. 
We have $W=|v_3|$ or $W=|v_4|$, because
\[
\log|v_2|< 5\log5 +28 \log28 - 4\log4 - 12\log6 - 8\log8 -10\log10 = 34.6412\dots<\log|v_1|. 
\]
The  values of the function $\mu(\omega)$ defined in section 6 are as follows: $\mu(\omega)=2$ if 
\[
\omega\in[1/7,1/6[\cup [2/9,1/4[\cup[2/7,3/10[\cup [3/7,1/2[
	\cup[4/7,3/5[\cup[5/7,3/4[\cup[6/7,7/8[;
\]
$\mu(\omega)=1$ if
\[
\begin{split}
\omega\in[1/9,1/7[\cup[1/6,2/9[\cup[1/4,2/7[\cup[3/10,3/7[ \\
		\cup[5/9,4/7[\cup[3/5,5/7[\cup[3/4,6/7[\cup[7/8,9/10[;
\end{split}
\]
$\mu(\omega)=0$ otherwise. 
Consequently $\delta=10.792110594854...\ $. By Theorem 7.1 we get 
\[
\mu_2(\log2)\leq 12.841618132152...,
\]
which improves upon the best known upper bound for $\mu_2(\log2)$, 
namely $15.651420...$, proved in~\cite{MarcovecchioLog2}.

By taking $m=2$, $n=3$, $p_1=3h-2j$, $p_2=2h-j$, $p_3=h$, 
$q_1=0$, $q_2=j-h$, $q_3=2(j-h)$ and $z=a/(a-b)$ we obtain 
all the nonquadraticity measures of $\log(a/b)$ on p. 151 and 182 
of~\cite{MarcovecchioLog2}.  Choosing $m=2$, $n=3$, $p_1=3h-2j$, 
$p_2=2h-j$, $p_3=h$, $q_1=0$, $q_2=j-h$, $q_3=2(j-h)$ and 
$z=-a-\sqrt{a(a+1)}$, by combining our results with Corollary 2.8 
of~\cite{MarcovecchioViola} we may also obtain all the $\Q(\sqrt{1+1/a})$--nonquadraticity 
measures of $\log(1+1/a)$ in Table 1 on p.267 of~\cite{MarcovecchioViola}, 
without using the $\C^2$--saddle point method.

We conclude with three more examples. We set $z=-4$, $m=3$, $n=4$, 
$p_1=6$, $p_2=7$, $p_3=8$, $p_4=9$, $q_1=10$, $q_2=11$, $q_3=12$, $q_4=13$. We obtain 
\[
P\bigg(\log\Big(\frac{5}{4}\Big)\bigg)\gg_\varepsilon H(P)^{-66.9403256794\dots-\varepsilon}
\]
for any $\varepsilon>0$ and any polynomial $P(x)$ of degree at most $3$. If we take 
$z=-5$, $m=3$, $n=4$, $p_1=8$, $p_2=9$, $p_3=10$, $p_4=11$, $q_1=12$, $q_2=13$,
$q_3=14$, $q_4=15$, we obtain 
\[
P\bigg(\log\Big(\frac{6}{5}\Big)\bigg)\gg_\varepsilon H(P)^{-36.9634662932\dots-\varepsilon} 
\qquad (\text{ Sorokin~\cite{SorokinLogs}: }  -154-\varepsilon)
\]
for any $\varepsilon>0$ and any polynomial $P(x)$ of degree at most $3$. By taking 
$z=-19$, $m=4$, $n=5$, $p_1=14$, $p_2=15$, $p_3=16$, $p_4=17$, $p_5=18$, $q_1=19$, 
$q_2=20$, $q_3=21$, $q_4=22$, $q_5=23$ we obtain 
\[
P\bigg(\log\Big(\frac{20}{19}\Big)\bigg)\gg_\varepsilon H(P)^{-565.5663269277\dots-\varepsilon}
\]
for any $\varepsilon>0$ and any polynomial $P(x)$ of degree at most $4$.
\section*{Acknowledgement}
The author is a Fellow ``Ing. Giorgio Schirillo'' of the Istituto Nazionale di Alta Matematica. 

\end{document}